\documentclass[11pt]{article}
\usepackage[a4paper,margin=1in]{geometry}
\usepackage{iftex}
\usepackage[T1]{fontenc}
\usepackage[utf8]{inputenc}
\usepackage{lmodern}
\usepackage{amsmath,amssymb,amsthm,mathtools,mathrsfs}
\usepackage{microtype}
\usepackage[numbers,square,sort&compress]{natbib}
\usepackage[unicode,hidelinks]{hyperref}
\usepackage{bookmark}
\newtheorem{theorem}{Theorem}[section]
\newtheorem{proposition}[theorem]{Proposition}
\newtheorem{lemma}[theorem]{Lemma}

\theoremstyle{remark}

\numberwithin{equation}{section}
\author{Yuhang Cao\thanks{Yuhang Cao, Hexiang Huang, Mengrui Yang and Sihuang Hu are with the State Key Laboratory of Cryptography and Digital Economy Security, the Key Laboratory of Cryptologic Technology and Information Security, Ministry of Education, and the School of Cyber Science and Technology, Shandong University, Qingdao, Shandong 266237, China. (email: yuhangcao@mail.sdu.edu.cn, hexianghuang@foxmail.com, 19553710279@163.com, husihuang@sdu.edu.cn).} \and Sihuang Hu \and Hexiang Huang \and Mengrui Yang}
\date{}

\title{The Norm Spectrum of Barnes--Wall Lattices Below Twice the Minimum}

\begin{document}

\maketitle

\begin{abstract}

Let \(\Lambda_m\) be the Barnes--Wall lattice of dimension \(2^m\), normalized to have minimum squared norm \(d_m:=2^{m-1}\). We determine its complete norm spectrum strictly between \(d_m\) and \(2d_m\), proving that, for every integer \(m\ge1\),
\[
\bigl\{\|x\|^2:x\in\Lambda_m,\ d_m<\|x\|^2<2d_m\bigr\}
=
\bigl\{2d_m-2^j:j\in\mathbb Z,\ \lceil m/2\rceil\le j\le m-2\bigr\}.
\]
In particular, the norm spectrum in this interval is nonempty for \(m\ge4\). The proof uses a Gaussian-integer model in which lattice automorphisms arising from Clifford unitaries transform a vector into a pair of coordinate subvectors in a lower-dimensional lattice. Divisibility of the differences of their squared norms and a matrix Parseval identity provide the arithmetic and averaging constraints needed for induction on \(m\). Every listed norm is realized explicitly by a sum of two minimal vectors of the lattice.

\end{abstract}

\section{Introduction}

The Barnes--Wall lattices form one of the best studied families of highly symmetric Euclidean lattices.
Originally introduced by \citet{BarnesWall1959}, they have found applications across several areas, including communication coding and decoding \citep{Forney1988PartII,CorlayEtAl2020,GrigorescuPeikert2014}, geometry and combinatorial design theory \citep{Bachoc2005}, and quantum information \citep{KliuchnikovSchoenenbeck2024,KalraSinha2026}.
Their automorphism groups and connections with Clifford groups were studied by \citet{Wall1962IV} and \citet{NebeRainsSloane2001}.
Barnes and Wall also determined their minimum norms and gave an explicit characterization and enumeration of their minimal vectors \citep{BarnesWall1959}.

We study the norms strictly between the minimum and twice the minimum of a Barnes--Wall lattice.
The construction of these lattices from nested binary Reed--Muller codes via Construction~D \citep[Section~IV.B]{Forney1988PartII} suggests a close connection between the short norm spectrum of the lattices and the low-weight structure of these codes.
\citet{BerlekampSloane1969} showed that, for a binary Reed--Muller code of minimum distance \(\delta\), every nonzero weight below \(2\delta\) has the form \(2\delta-2^j\) for an integer \(j\).
\citet{KasamiTokura1970} subsequently classified the codewords in this range and determined the corresponding weight enumerators.
As reported by \citet{Nebe2026}, Keller conjectured an analogous restriction for Barnes--Wall lattices, motivated by this pattern, computational evidence, and the known theta series in dimensions \(64\) and \(128\).

We write \(\Lambda_m\) for the Barnes--Wall lattice of dimension \(2^m\), normalized so that \(d_m:=\min(\Lambda_m)=2^{m-1}\).
Here and throughout, the norm of a vector means its squared Euclidean length.
Our main result determines the complete norm spectrum of \(\Lambda_m\) in the open interval \((d_m,2d_m)\).

\begin{theorem}\label{thm:main-spectrum}
For every integer \(m\ge1\),
\begin{equation}\label{eq:main-spectrum}
\bigl\{\lVert x\rVert^2:x\in\Lambda_m,\ d_m<\lVert x\rVert^2<2d_m\bigr\}
=
\bigl\{2d_m-2^j:j\in\mathbb Z,\ \lceil m/2\rceil\le j\le m-2\bigr\}.
\end{equation}
In particular, this spectrum is empty for \(m\le3\), while every displayed value occurs for \(m\ge4\).
\end{theorem}

\citet[Theorem~3.3]{Nebe2026} proved that no vector of \(\Lambda_m\) has norm strictly between \(d_m\) and \(3d_m/2\).
For \(m\ge4\), Theorem~\ref{thm:main-spectrum} shows that \(3d_m/2\) is the smallest norm above \(d_m\) and determines every subsequent norm below \(2d_m\).

Our proof develops the recursive approach of \citet[Section~3]{KliuchnikovSchoenenbeck2024}, who characterize the minimal vectors of Barnes--Wall lattices as stabilizer states.
We work with a Gaussian-integer model \(L_{m-1}\), whose coordinatewise realification is isometric to \(\Lambda_{m}\), as established in Section~\ref{sec:lattice-models}.
We use lattice automorphisms arising from Clifford unitaries to obtain recursive decompositions.
Specifically, after applying such an automorphism, we write a vector in \(L_{m-1}\) as a pair \((u,v)\), where \(u,v\in L_{m-2}\).
The subsequent analysis requires additional control of the norms of these coordinate subvectors, which we obtain through 2-divisibility and a matrix Parseval identity.

We now outline how these ingredients constrain the possible norms.
For \(m\ge2\), each nonidentity tensor product \(R\) of \(m-1\) Pauli matrices admits a Clifford unitary \(C\) that is an automorphism of \(L_{m-1}\). For any \(x\in L_{m-1}\), the coordinate subvectors in \(Cx=(u_R,v_R)\) belong to \(L_{m-2}\) and satisfy
\[
\|u_R\|^2=\frac{\|x\|^2+q_R(x)}{2},
\qquad
\|v_R\|^2=\frac{\|x\|^2-q_R(x)}{2},
\]
where \(q_R(x):=\langle x,Rx\rangle\) is the difference between the squared norms of \(u_R\) and \(v_R\).

For \(x\in L_{m-1}\) with \(d_m<\|x\|^2<2d_m\), 2-divisibility gives a uniform lower bound for the nonzero values of \(q_R(x)^2\). We compare this bound with the average of \(q_R(x)^2\) over all nonidentity \(R\), which is determined exactly by a matrix Parseval identity.
When \(m\) is odd, this average lies below the lower bound, forcing some \(q_R(x)=0\) and hence an equal-norm decomposition to which the induction applies.
For even \(m\), combining the averaging bound with the induction hypothesis leaves at most one candidate norm beyond those listed in Theorem 1.1. We rule out this possibility by considering the square-sum identity modulo \(3\).

We realize the candidate norms using the affine-subspace description of \(\Lambda_m\).
For each exponent \(j\) in Theorem~\ref{thm:main-spectrum}, we construct two minimal vectors \(p,q\) supported on suitable affine subspaces and satisfying \((p,q)=-2^{j-1}\), so that \(\|p+q\|^2=2d_m-2^j\).

The paper is organized as follows.
Section~\ref{sec:lattice-models} introduces the lattice models and their basic arithmetic properties.
Sections~\ref{sec:clifford-decompositions} and~\ref{sec:dyadic-parseval} establish the formulas for the norms of the coordinate subvectors, divisibility, and the matrix Parseval identity used in the proof.
Sections~\ref{sec:attainment} and~\ref{sec:exclusion} prove Theorem~\ref{thm:main-spectrum} by constructing the listed norms and excluding all others, respectively.
Finally, Section~\ref{sec:conclusion} summarizes our results and discusses possible directions for future work.

\section{Barnes--Wall Lattices over \texorpdfstring{\(\mathbb Z\)}{Z} and \texorpdfstring{\(\mathbb Z[i]\)}{Z[i]}}
\label{sec:lattice-models}

We recall two realizations of the Barnes--Wall lattices and the basic
arithmetic properties used below.
The affine-subspace realization is used to construct vectors of the
required norms, whereas the Gaussian-integer realization is used in excluding other norms.
We connect the two realizations by an isometric realification map, which allows norm properties to be transferred between the two models.

\subsection{The \texorpdfstring{\(\mathbb Z\)}{Z}-Lattice \texorpdfstring{\(\Lambda_m\)}{Lambda m}}
\label{subsec:real-lattice}

For an integer \(m\ge1\), put \(V_m:=\mathbb F_2^m\), and let
\((e_v)_{v\in V_m}\) be the standard orthonormal basis of
\(\mathbb R^{V_m}\).
Define the standard Euclidean lattice
\[
\Gamma_m:=\bigoplus_{v\in V_m}\mathbb Z e_v.
\]
For a subset \(S\subseteq V_m\), write
\[
x_S:=\sum_{v\in S}e_v
\]
for its characteristic vector.

We follow the notation of
\citet[Section~2]{Nebe2026}
for the affine-subspace construction originating in
\citet[Sections~3--4]{BarnesWall1959}.
For every \(m\ge1\), define the real Barnes--Wall lattice
\(\Lambda_m\) and its companion \(\Delta_m\) by
\begin{align}
\Lambda_m
&:=\left\langle
2^{\lfloor(m-\dim U)/2\rfloor}x_U:
U\subseteq V_m\text{ affine}
\right\rangle_{\mathbb Z},
\label{eq:Lambda-affine}\\
\Delta_m
&:=\left\langle
2^{\lfloor(m-\dim U+1)/2\rfloor}x_U:
U\subseteq V_m\text{ affine}
\right\rangle_{\mathbb Z}.
\label{eq:Delta-affine}
\end{align}
To connect the affine-subspace realization with the Gaussian-integer
model introduced below, we use the following recursive descriptions
of \(\Lambda_m\) and \(\Delta_m\).
Equation~\eqref{eq:Lambda-recursion} is a restatement of the recursion in
\citet[Theorem~3.2]{Nebe2026}, written with our coordinate grouping.
Equation~\eqref{eq:Delta-recursion} follows from this description
and the transformation in
\citet[Theorem~4.3]{BarnesWall1959}.

\begin{proposition}\label{prop:real-recursion}
For \(m\ge2\), the decomposition
\(V_m=\mathbb F_2\times V_{m-1}\) identifies
\[
\Gamma_m\cong\Gamma_{m-1}\perp\Gamma_{m-1},
\]
with vectors written as pairs \((u,v)\). Then
\begin{align}
\Lambda_m
&=\{(u,v)\in\Lambda_{m-1}^{\,2}:u-v\in\Delta_{m-1}\},
\label{eq:Lambda-recursion}\\
\Delta_m
&=\{(u,v)\in\Delta_{m-1}^{\,2}:u-v\in2\Lambda_{m-1}\},
\label{eq:Delta-recursion}
\end{align}
with initial conditions
\[
\Lambda_1=\mathbb Z^2,
\qquad
\Delta_1=\{(a,b)\in\mathbb Z^2:a\equiv b\pmod2\}.
\]
\end{proposition}

\begin{proof}
The initial conditions and the inclusions
\[
2\Lambda_m\subseteq\Delta_m\subseteq\Lambda_m
\]
follow directly from the defining generators.
For \(m\ge2\), identify the two coordinate subvectors in \citet[Theorem~3.2]{Nebe2026} with those indexed by \(\{0\}\times V_{m-1}\) and \(\{1\}\times V_{m-1}\), respectively.
Then that theorem gives
\[
\Lambda_m
=\{(\ell+t,\ell):
\ell\in\Lambda_{m-1},\
t\in\Delta_{m-1}\}.
\]
Since \(\Delta_{m-1}\subseteq\Lambda_{m-1}\), this is
equivalent to \eqref{eq:Lambda-recursion}.
To derive the corresponding recursion for \(\Delta_m\), we apply
the transformation in
\citet[Theorem~4.3]{BarnesWall1959}.
With the same coordinate grouping, this transformation is
\[
H_m(a,b)=(a+b,a-b)
\]
and satisfies
\[
H_m\Lambda_m=\Delta_m,
\qquad
\|H_mx\|^2=2\|x\|^2.
\]
Applying \(H_m\) to the preceding expression yields
\[
\begin{aligned}
\Delta_m
&=\{(2\ell+t,t):
\ell\in\Lambda_{m-1},\
t\in\Delta_{m-1}\}\\
&=\{(u,v)\in\Delta_{m-1}^{\,2}:
u-v\in2\Lambda_{m-1}\}.
\end{aligned}
\]
This proves \eqref{eq:Delta-recursion} and completes the proof.
\end{proof}

\subsection{The Hermitian \texorpdfstring{\(\mathbb Z[i]\)}{Z[i]}-Lattice \texorpdfstring{\(L_n\)}{L n}}

We use the basis matrix of
\citet[Section~2.5]{KliuchnikovSchoenenbeck2024}
and specialize their construction to \(E=\mathbb Q(i)\),
whose ring of integers is \(\mathbb Z[i]\).
We denote the resulting lattice by \(L_n\).
Its compatibility with the real lattices defined above will be
established in Section~\ref{subsec:recursion-realification}.
For the earlier squaring construction of Barnes--Wall lattices,
see \citet[Section~IV.B]{Forney1988PartII}.

Put
\[
\pi:=1+i,\qquad
B:=\begin{pmatrix}\pi&1\\0&1\end{pmatrix},
\]
and, for every integer \(n\ge0\), define
\begin{equation}\label{eq:Gaussian-lattice-definition}
L_n:=B^{\otimes n}\mathbb Z[i]^{2^n}\subseteq\mathbb C^{2^n}.
\end{equation}
Here we take \(B^{\otimes0}=(1)\), so that \(L_0=\mathbb Z[i]\).

\citet[Definition~1]{MicciancioNicolosi2008}
define the Gaussian lattice using the rows of
\(M^{\otimes n}\), where
\[
M=\begin{pmatrix}1&1\\0&\pi\end{pmatrix}.
\]
After converting to column generators, the two descriptions agree since
\[
B=M^{\mathsf T}Q,
\qquad
Q=\begin{pmatrix}\pi&1\\-1&0\end{pmatrix}
\in\mathrm{GL}_2(\mathbb Z[i]),
\qquad \det Q=1.
\]
Tensoring this identity gives the same lattice, without rescaling.

We also equip \(\mathbb C^{2^n}\) with the standard Hermitian form
\(\langle x,y\rangle:=x^\dagger y=\sum_{k=1}^{2^n}\overline{x_k}y_k\),
where \(\dagger\) denotes conjugate transpose. The associated squared norm is \(\|x\|^2:=\langle x,x\rangle=\sum_{k=1}^{2^n}|x_k|^2\).
For \(E=\mathbb Q(i)\), this squared norm agrees with the
norm convention in
\citet[Section~2.2]{KliuchnikovSchoenenbeck2024}.

\subsection{Recursive Structure and Realification}
\label{subsec:recursion-realification}

We now identify the Gaussian model with the real lattices
defined in Section~\ref{subsec:real-lattice}.
The tensor-product construction yields recursions for \(L_n\)
and \(\pi L_n\) that, under coordinatewise realification,
agree with the simultaneous recursion in
Proposition~\ref{prop:real-recursion}.

For \(n\ge1\), write a vector in \(\mathbb C^{2^n}\), with coordinates indexed in tensor order by \(\mathbb F_2^n\), as \((u,v)\). Here \(u\) and \(v\) are the coordinate subvectors indexed by \(\{0\}\times\mathbb F_2^{n-1}\) and \(\{1\}\times\mathbb F_2^{n-1}\), respectively, each with \(2^{n-1}\) coordinates.
Grouping coordinates according to this first tensor factor gives
\[
B^{\otimes n}
=
\begin{pmatrix}
\pi B^{\otimes(n-1)} & B^{\otimes(n-1)}\\
0 & B^{\otimes(n-1)}
\end{pmatrix}.
\]
Consequently,
\[
L_n=\{(\pi a+b,b):a,b\in L_{n-1}\}.
\]
Equivalently,
\begin{equation}\label{eq:Gaussian-recursion}
L_n
=\{(u,v)\in L_{n-1}^2:
u-v\in\pi L_{n-1}\}.
\end{equation}
With \(\phi=\pi\), this is equivalent to the parametrization
\(BW_n=\{[u,u+\phi v]:u,v\in BW_{n-1}\}\), with \(BW_0=\mathbb Z[i]\),
used by \citet[Definition~1.1]{GrigorescuPeikert2014}.

Multiplying the parametrization \(L_n=\{(\pi a+b,b):a,b\in L_{n-1}\}\)
by \(\pi\), and using \(\pi^2=2i\) and \(iL_{n-1}=L_{n-1}\), we obtain
\[
\pi L_n
=\{(c+2a,c):a\in L_{n-1},\ c\in\pi L_{n-1}\}.
\]
Since \(2L_{n-1}\subseteq\pi L_{n-1}\), this is equivalently
\begin{equation}\label{eq:pi-Gaussian-recursion}
\pi L_n
=\{(u,v)\in\pi L_{n-1}^2:
u-v\in2L_{n-1}\}.
\end{equation}

For \(n\ge0\), index the coordinates of \(\mathbb C^{2^n}\) by \(V_n=\mathbb F_2^n\) in tensor order.
Identifying \(V_{n+1}=V_n\times\mathbb F_2\), define the coordinatewise realification
\begin{equation}\label{eq:coordinatewise-realification}
\rho_n:\mathbb C^{V_n}\longrightarrow\mathbb R^{V_{n+1}},
\qquad
\rho_n((z_a)_{a\in V_n})
:=\sum_{a\in V_n}
\bigl(\operatorname{Re}(z_a)e_{(a,0)}
+\operatorname{Im}(z_a)e_{(a,1)}\bigr).
\end{equation}
With this ordering, the decomposition according to the first tensor factor satisfies
\begin{equation}\label{eq:realification-respects-blocks}
\rho_n(u,v)
=\bigl(\rho_{n-1}(u),\rho_{n-1}(v)\bigr)
\qquad(n\ge1).
\end{equation}

The following lemma identifies the two realizations in our specified
coordinate order and verifies that no rescaling is required.
\begin{lemma}
\label{lem:Gaussian-realization}
For every \(n\ge0\), the map
\[
\rho_n:\mathbb C^{V_n}\longrightarrow\mathbb R^{V_{n+1}}
\]
is a real-linear isometry satisfying
\begin{equation}\label{eq:Gaussian-real-identifications}
\rho_n(L_n)=\Lambda_{n+1},
\qquad
\rho_n(\pi L_n)=\Delta_{n+1}.
\end{equation}
\end{lemma}

\begin{proof}
The map \(\rho_n\) is a real-linear bijection, with inverse obtained by sending each coordinate pair \((x_{(a,0)},x_{(a,1)})\) to \(x_{(a,0)}+i x_{(a,1)}\).
It is an isometry since
\[
\|\rho_n(z)\|^2
=\sum_{a\in V_n}
\bigl((\operatorname{Re}z_a)^2+(\operatorname{Im}z_a)^2\bigr)
=\sum_{a\in V_n}|z_a|^2
=\|z\|^2.
\]

We prove the two lattice identifications simultaneously by induction on \(n\). For \(n=0\), we have
\[
\rho_0(L_0)=\rho_0(\mathbb Z[i])=\mathbb Z^2=\Lambda_1.
\]
Since
\[
(1+i)(a+ib)=(a-b)+i(a+b),
\]
we also obtain
\[
\begin{aligned}
\rho_0(\pi L_0)
&=\{(a-b,a+b):a,b\in\mathbb Z\}\\
&=\{(r,s)\in\mathbb Z^2:r\equiv s\pmod2\}
=\Delta_1.
\end{aligned}
\]

Now let \(n\ge1\) and assume that
\[
\rho_{n-1}(L_{n-1})=\Lambda_n,
\qquad
\rho_{n-1}(\pi L_{n-1})=\Delta_n.
\]
For \(a,b\in L_{n-1}\), real-linearity and injectivity of \(\rho_{n-1}\) give
\[
a-b\in\pi L_{n-1}
\quad\Longleftrightarrow\quad
\rho_{n-1}(a)-\rho_{n-1}(b)\in\Delta_n.
\]
Since \(\rho_{n-1}\) maps \(L_{n-1}\) bijectively onto \(\Lambda_n\), the Gaussian recursion therefore yields
\[
\begin{aligned}
\rho_n(L_n)
&=\{(u,v)\in\Lambda_n^2:
u-v\in\Delta_n\}\\
&=\Lambda_{n+1}.
\end{aligned}
\]
Similarly, real-linearity gives
\[
\rho_{n-1}(2L_{n-1})=2\Lambda_n.
\]
Thus, for \(a,b\in\pi L_{n-1}\),
\[
a-b\in2L_{n-1}
\quad\Longleftrightarrow\quad
\rho_{n-1}(a)-\rho_{n-1}(b)\in2\Lambda_n.
\]
Using the recursion for \(\pi L_n\), we obtain
\[
\begin{aligned}
\rho_n(\pi L_n)
&=\{(u,v)\in\Delta_n^2:
u-v\in2\Lambda_n\}\\
&=\Delta_{n+1}.
\end{aligned}
\]
The final equalities in the two calculations follow from
\eqref{eq:Lambda-recursion} and \eqref{eq:Delta-recursion},
respectively.
This completes the induction.
\end{proof}

\subsection{Minimum and Norm Divisibility}
We now record the minimum and norm divisibility of \(L_n\),
which will be used in the subsequent proofs.
We first recall the corresponding results for the real Barnes--Wall
lattices and then transfer them to the Gaussian-integer model via
the realification isometry of Lemma~\ref{lem:Gaussian-realization}.

For the real lattice defined by \eqref{eq:Lambda-affine}, the minimum
goes back to the original calculation of
\citet[Theorem~3.2]{BarnesWall1959}.
Specializing their general minimum formula to
\(\lambda_r=\lfloor r/2\rfloor\) gives
\begin{equation}\label{eq:Lambda-minimum}
\min(\Lambda_m)=2^{m-1}
\qquad(m\ge1).
\end{equation}
See also \citet[Theorem~2.3]{Nebe2026} for a restatement in the same
normalization.

The norm divisibility can likewise be deduced from the affine generators
and the duality relation in \citet[Theorem~3.1]{BarnesWall1959}, 
which expresses the evenness of the real
Barnes--Wall lattice.
Indeed, the rescaled lattice
\(2^{-\frac12\lfloor(m-1)/2\rfloor}\Lambda_m\) is even for \(m\ge2\);
see \citet[Remark~2.4]{Nebe2026}.
Rescaling squared norms therefore gives
\begin{equation}\label{eq:Lambda-norm-divisibility}
\|y\|^2\in2^{\lceil m/2\rceil}\mathbb Z
\qquad(m\ge2,\ y\in\Lambda_m).
\end{equation}
This divisibility is also recorded by
\citet*[p.~1739]{HeningerRainsSloane2006} in the same normalization.

The realification isometry then transfers these properties to the Gaussian model.

\begin{lemma}\label{lem:Ln-minimum}
For every \(n\ge0\),
\begin{equation}\label{eq:Ln-minimum}
\min(L_n)=2^n.
\end{equation}
\end{lemma}
\begin{proof}
Lemma~\ref{lem:Gaussian-realization} and \eqref{eq:Lambda-minimum} give
\(\min(L_n)=\min(\Lambda_{n+1})=2^n\).
\end{proof}

\begin{lemma}\label{lem:Ln-norm-divisibility}
For every \(n\ge1\) and \(x\in L_n\),
\begin{equation}\label{eq:Ln-norm-divisibility}
\|x\|^2\in2^{\lceil(n+1)/2\rceil}\mathbb Z.
\end{equation}
\end{lemma}
\begin{proof}
By Lemma~\ref{lem:Gaussian-realization},
\(\rho_n(x)\in\Lambda_{n+1}\) and
\(\|\rho_n(x)\|^2=\|x\|^2\).
Since \(n+1\ge2\), the conclusion follows from
\eqref{eq:Lambda-norm-divisibility} with \(m=n+1\).
\end{proof}

\section{Lattice Automorphisms and Recursive Decompositions}
\label{sec:clifford-decompositions}

\subsection{Clifford Unitaries as Lattice Automorphisms}

Let
\[
I=\begin{pmatrix}1&0\\0&1\end{pmatrix},
\qquad
X=\begin{pmatrix}0&1\\1&0\end{pmatrix},
\qquad
Y=\begin{pmatrix}0&-i\\i&0\end{pmatrix},
\qquad
Z=\begin{pmatrix}1&0\\0&-1\end{pmatrix}
\]
be the Pauli matrices. For \(n\ge0\), define
\[
\mathscr E_n
:=\{P_1\otimes\cdots\otimes P_n:
P_j\in\{I,X,Y,Z\}\}.
\]
Every \(R\in\mathscr E_n\) satisfies \(R^\dagger=R\) and \(R^2=I\), and is therefore both Hermitian and unitary.

In the standard Pauli-group formulation of \citet[Section~II]{DehaeneDeMoor2003}, a unitary \(C\) is called a Clifford unitary if
\[
CRC^\dagger\in\{\pm P:P\in\mathscr E_n\}
\qquad\text{for every }R\in\mathscr E_n.
\]
A complex-linear unitary map \(C\) is an automorphism of \(L_n\) if and only if \(CL_n=L_n\).
Following \citet[Sections~2.3.3 and~2.5.3]{KliuchnikovSchoenenbeck2024}, we use the phase gate, the phase-adjusted Hadamard gate, and the controlled-NOT gate to construct Clifford unitaries that are automorphisms of \(L_n\).
The following lemma shows that every nonidentity
\(R\in\mathscr E_n\) can be conjugated to
\(Z\otimes I^{\otimes(n-1)}\) by a Clifford unitary
that is an automorphism of \(L_n\).
This will allow us to use the lattice recursion to express
the difference of the squared norms of two coordinate subvectors in \(L_{n-1}\).

\begin{lemma}
\label{lem:tensor-normalization}
Let \(n\ge1\). For every \(R\in\mathscr E_n\setminus\{I\}\), there exists a Clifford unitary \(C\) such that
\begin{equation}\label{eq:tensor-normalization}
CL_n=L_n,
\qquad
CRC^\dagger=Z_1,
\end{equation}
where
\[
Z_1:=Z\otimes I^{\otimes(n-1)}.
\]
\end{lemma}

\begin{proof}
Put \(G_n:=B^{\otimes n}\), so that \(L_n=G_n\mathbb Z[i]^{2^n}\).
For an invertible complex-linear map \(U\), we have
\[
UL_n=L_n
\quad\Longleftrightarrow\quad
(G_n^{-1}UG_n)\mathbb Z[i]^{2^n}
=\mathbb Z[i]^{2^n}.
\]
Thus \(UL_n=L_n\) if and only if
\begin{equation}\label{eq:lattice-automorphism-criterion}
G_n^{-1}UG_n\in\operatorname{GL}_{2^n}(\mathbb Z[i]).
\end{equation}
Equivalently, both \(G_n^{-1}UG_n\) and its inverse must have Gaussian-integer entries.
When \(U\) is unitary, this is exactly the condition that \(U\) is an automorphism of \(L_n\).
We apply this criterion to the generators used below.

Set
\[
S:=\begin{pmatrix}1&0\\0&i\end{pmatrix},
\qquad
\widetilde H:=\frac1\pi
\begin{pmatrix}1&1\\1&-1\end{pmatrix}.
\]
Both matrices are unitary, and direct calculation gives
\begin{equation}\label{eq:adjusted-HS-Gaussian-conjugates}
B^{-1}SB=
\begin{pmatrix}1&-i\\0&i\end{pmatrix},
\qquad
B^{-1}\widetilde HB=
\begin{pmatrix}0&-i\\1&0\end{pmatrix}.
\end{equation}
These matrices have Gaussian-integer entries and determinant \(i\), so they belong to \(\operatorname{GL}_2(\mathbb Z[i])\).
Hence \(S\) and \(\widetilde H\) are automorphisms of \(L_1\). Their conjugation actions are
\[
\begin{aligned}
SXS^\dagger&=Y,&
SYS^\dagger&=-X,&
SZS^\dagger&=Z,\\
\widetilde HX\widetilde H^\dagger&=Z,&
\widetilde HY\widetilde H^\dagger&=-Y,&
\widetilde HZ\widetilde H^\dagger&=X.
\end{aligned}
\]
Thus both are Clifford unitaries.

Let \(T:=\operatorname{CX}_{1\to2}\) denote the controlled-NOT gate, defined on the standard coordinate vectors of \(\mathbb C^4\) by
\[
T e_{(a,b)}=e_{(a,a+b)}
\qquad(a,b\in\mathbb F_2).
\]
This is a unitary permutation matrix. In the coordinate order \((0,0),(0,1),(1,0),(1,1)\), direct calculation gives
\begin{equation}\label{eq:CX-Gaussian-conjugate}
G_2^{-1}TG_2
=
\begin{pmatrix}
1&0&\overline\pi&0\\
0&1&-1&0\\
0&0&-1&0\\
0&0&\pi&1
\end{pmatrix}.
\end{equation}
The displayed matrix has Gaussian-integer entries and determinant \(-1\), so \(T\) is an automorphism of \(L_2\). Moreover,
\[
\begin{aligned}
T(X\otimes I)T^\dagger&=X\otimes X,&
T(I\otimes X)T^\dagger&=I\otimes X,\\
T(Z\otimes I)T^\dagger&=Z\otimes I,&
T(I\otimes Z)T^\dagger&=Z\otimes Z.
\end{aligned}
\]
Together with \(Y=iXZ\), these identities show that \(T\) is a Clifford unitary.

Let \(U\) be one of the preceding gates acting on \(k\in\{1,2\}\) tensor factors. For \(n\ge k\), we have
\[
G_n^{-1}\bigl(U\otimes I^{\otimes(n-k)}\bigr)G_n
=(G_k^{-1}UG_k)\otimes I^{\otimes(n-k)}
\in\operatorname{GL}_{2^n}(\mathbb Z[i]).
\]
Thus extending \(U\) by identity factors gives an automorphism of \(L_n\), and the extended operator remains a Clifford unitary.
Moreover, any permutation \(P\) of the tensor factors commutes with \(G_n=B^{\otimes n}\), so
\[
G_n^{-1}PG_n=P\in\operatorname{GL}_{2^n}(\mathbb Z[i]).
\]
Such a permutation is also Clifford, since it permutes the factors of each Pauli tensor product.
Conjugating the extended gates by these permutations therefore gives Clifford unitaries \(S\) and \(\widetilde H\) on any single factor, and controlled-NOT gates on any ordered pair of distinct factors. These operators are automorphisms of \(L_n\).
Their inverses have the same properties.

Write
\[
R=P_1\otimes\cdots\otimes P_n,
\qquad P_j\in\{I,X,Y,Z\},
\]
and let \(A:=\{j:P_j\ne I\}\). Since \(R\ne I\), the set \(A\) is nonempty. Define
\[
D_j:=
\begin{cases}
I,&P_j\in\{I,Z\},\\
\widetilde H,&P_j=X,\\
\widetilde H S^\dagger,&P_j=Y,
\end{cases}
\qquad
D:=D_1\otimes\cdots\otimes D_n.
\]
By the preceding argument, \(D\) is both a Clifford unitary and an automorphism of \(L_n\). The identities
\[
\widetilde H X\widetilde H^\dagger=Z,
\qquad
S^\dagger YS=X
\]
give
\begin{equation}\label{eq:local-Pauli-to-Z}
DRD^\dagger=Z_A:=\prod_{j\in A}Z_j,
\end{equation}
where \(Z_j:=I^{\otimes(j-1)}\otimes Z\otimes I^{\otimes(n-j)}\).

Since tensor-factor permutations are both Clifford unitaries and automorphisms of \(L_n\), we may assume, after such a permutation, that \(1\in A\).
For each \(j\in A\setminus\{1\}\), conjugation by \(\operatorname{CX}_{j\to1}\) fixes \(Z_j\), sends \(Z_1\) to \(Z_jZ_1\), and leaves every other \(Z_k\) unchanged.
Since \(Z_j^2=I\), it follows that
\begin{equation}\label{eq:CX-Z-elimination}
\operatorname{CX}_{j\to1}Z_A
\operatorname{CX}_{j\to1}^{\dagger}
=Z_{A\setminus\{j\}}.
\end{equation}
Applying these gates successively for all \(j\in A\setminus\{1\}\) therefore reduces \(Z_A\) to \(Z_1\).

Let \(C\) be the composite obtained by first applying \(D\), then the tensor-factor permutation if needed, and finally the successive controlled-NOT gates described above.
Each operation is both a Clifford unitary and an automorphism of \(L_n\), so \(C\) is a Clifford unitary satisfying
\[
CL_n=L_n,
\qquad
CRC^\dagger=Z_1.
\]
This proves the lemma.
\end{proof}

\subsection{Recursive Decompositions}

For \(n\ge0\), \(x\in L_n\), and \(R\in\mathscr E_n\), define
\begin{equation}\label{eq:imbalance-scalar}
q_R(x):=\langle x,Rx\rangle.
\end{equation}
Since \(R\) is Hermitian, \(q_R(x)\) is real.
Moreover, \(x\) has Gaussian-integer coordinates and \(R\) has Gaussian-integer entries, so \(q_R(x)\in\mathbb Z[i]\).
Hence
\[
q_R(x)\in\mathbb Z.
\]
For the identity matrix, we have \(q_I(x)=\|x\|^2\).

\begin{lemma}\label{lem:recursive-block-norms}
Let \(n\ge1\), \(x\in L_n\), and \(R\in\mathscr E_n\setminus\{I\}\).
There exist \(u_R,v_R\in L_{n-1}\) such that
\[
u_R-v_R\in\pi L_{n-1}
\]
and
\begin{equation}\label{eq:split-norms}
\|u_R\|^2=\frac{\|x\|^2+q_R(x)}2,
\qquad
\|v_R\|^2=\frac{\|x\|^2-q_R(x)}2.
\end{equation}
\end{lemma}

\begin{proof}
Let \[Z_1=Z\otimes I^{\otimes(n-1)}.\]
By Lemma~\ref{lem:tensor-normalization}, choose a Clifford unitary \(C\) such that
\[
CL_n=L_n,
\qquad
CRC^\dagger=Z_1.
\]
Write \(Cx=(u_R,v_R)\) according to the first tensor factor.
Since \(Cx\in L_n\), the Gaussian recursion gives
\[
u_R,v_R\in L_{n-1},
\qquad
u_R-v_R\in\pi L_{n-1}.
\]

Since \(C\) is unitary and \(CRC^\dagger=Z_1\), we have
\begin{equation}\label{eq:block-norm-difference}
\begin{aligned}
q_R(x)
&=\langle x,Rx\rangle\\
&=\langle Cx,CRx\rangle\\
&=\langle Cx,Z_1Cx\rangle\\
&=\langle(u_R,v_R),(u_R,-v_R)\rangle\\
&=\|u_R\|^2-\|v_R\|^2.
\end{aligned}
\end{equation}

Since \(C\) is unitary and the two coordinate subspaces are orthogonal, we also have
\begin{equation}\label{eq:block-norm-sum}
\|x\|^2=\|Cx\|^2
=\|(u_R,v_R)\|^2
=\|u_R\|^2+\|v_R\|^2.
\end{equation}

Adding and subtracting the preceding two identities gives
\[
\|u_R\|^2=\frac{\|x\|^2+q_R(x)}2,
\qquad
\|v_R\|^2=\frac{\|x\|^2-q_R(x)}2.
\]
This proves the lemma.
\end{proof}

In particular, if \(q_R(x)=0\), then the two coordinate subvectors have equal squared norm:
\[
\|u_R\|^2=\|v_R\|^2=\frac{\|x\|^2}{2}.
\]

\section{Norm Divisibility and the Matrix Parseval Identity}
\label{sec:dyadic-parseval}

\subsection{Congruences for the Norms of Coordinate Subvectors}

The condition \(u_R-v_R\in\pi L_{n-1}\) places \(u_R\) and \(v_R\) in the same coset of \(\pi L_{n-1}\) in \(L_{n-1}\).
We first establish a norm congruence for vectors in the same coset, and then apply it to show that \(q_R(x)\) is divisible by a power of \(2\).

\begin{lemma}\label{lem:Ln-class-congruence}
Let \(r\ge0\) and \(u,v\in L_r\). If \(u-v\in\pi L_r\), then
\begin{equation}\label{eq:Ln-class-congruence}
\|u\|^2-\|v\|^2
\in2^{\lfloor(r+1)/2\rfloor+1}\mathbb Z.
\end{equation}
\end{lemma}

\begin{proof}
Choose \(w\in L_r\) such that \(u=v+\pi w\). Since \(|\pi|^2=2\),
\[
\|u\|^2-\|v\|^2
=2\operatorname{Re}\!\bigl(\pi\langle v,w\rangle\bigr)
+2\|w\|^2.
\]
If \(r=0\), then \(v,w\in\mathbb Z[i]\), so both terms on the right are even integers, proving the claim. Hence assume \(r\ge1\).

For the first term, observe that the one-factor Gram matrix
\[
B^\dagger B
=\begin{pmatrix}2&\overline\pi\\ \pi&2\end{pmatrix}
\]
has entries in \(\pi\mathbb Z[i]\). Therefore
\[
(B^{\otimes r})^\dagger B^{\otimes r}
=(B^\dagger B)^{\otimes r}
\]
has entries in \(\pi^r\mathbb Z[i]\). Since \(v\) and \(w\) are Gaussian-integer linear combinations of the columns of \(B^{\otimes r}\), it follows that
\(\langle v,w\rangle\in\pi^r\mathbb Z[i]\).
Using \(\pi^2=2i\), we obtain
\[
\pi\langle v,w\rangle
\in\pi^{r+1}\mathbb Z[i]
\subseteq2^{\lfloor(r+1)/2\rfloor}\mathbb Z[i],
\]
and hence
\[
2\operatorname{Re}\!\bigl(\pi\langle v,w\rangle\bigr)
\in2^{\lfloor(r+1)/2\rfloor+1}\mathbb Z.
\]

For the second term, Lemma~\ref{lem:Ln-norm-divisibility} gives
\[
2\|w\|^2
\in2^{\lceil(r+1)/2\rceil+1}\mathbb Z
\subseteq2^{\lfloor(r+1)/2\rfloor+1}\mathbb Z.
\]
Thus both terms in the norm expansion are divisible by
\(2^{\lfloor(r+1)/2\rfloor+1}\), proving the lemma.
\end{proof}

For \(m\ge2\), put
\begin{equation}\label{eq:gamma-definition}
\gamma_m:=2^{\lceil m/2\rceil}.
\end{equation}
Let \(x\in L_{m-1}\) and \(R\in\mathscr E_{m-1}\setminus\{I\}\).
By Lemma~\ref{lem:recursive-block-norms}, there exist \(u_R,v_R\in L_{m-2}\) satisfying
\[
u_R-v_R\in\pi L_{m-2},
\qquad
q_R(x)=\|u_R\|^2-\|v_R\|^2.
\]
Applying Lemma~\ref{lem:Ln-class-congruence} with \(r=m-2\), we obtain
\begin{equation}\label{eq:imbalance-divisibility}
q_R(x)\in
2^{\lfloor(m-1)/2\rfloor+1}\mathbb Z
=\gamma_m\mathbb Z,
\end{equation}
where we used \(\lfloor(m-1)/2\rfloor+1=\lceil m/2\rceil\).

Consequently, for every \(R\in\mathscr E_{m-1}\setminus\{I\}\),
\[
q_R(x)\ne0
\quad\Longrightarrow\quad
|q_R(x)|\ge\gamma_m
\quad\Longrightarrow\quad
q_R(x)^2\ge\gamma_m^2.
\]
We will compare this lower bound with the average of \(q_R(x)^2\) over all nonidentity \(R\).

\subsection{The Matrix Parseval Identity}

The required square-sum identity of \(q_R(x)\) is an instance of Parseval's identity in \(M_{2^n}(\mathbb C)\).
We obtain it by expanding \(xx^\dagger\) in the orthogonal basis of Pauli tensor products with respect to the Hilbert--Schmidt inner product.

For \(n\ge0\), equip \(M_{2^n}(\mathbb C)\) with the Hilbert--Schmidt inner product
\[
\langle A,B\rangle_{\mathrm{HS}}
:=\operatorname{Tr}(A^\dagger B).
\]
The one-factor Pauli matrices satisfy
\[
\operatorname{Tr}(P^\dagger Q)=2\delta_{P,Q}
\qquad(P,Q\in\{I,X,Y,Z\}).
\]
By multiplicativity of the trace on tensor products, it follows that
\begin{equation}\label{eq:Pauli-trace-orthogonality}
\operatorname{Tr}(R^\dagger S)=2^n\delta_{R,S}
\qquad(R,S\in\mathscr E_n).
\end{equation}
Since \(|\mathscr E_n|=4^n=\dim_{\mathbb C}M_{2^n}(\mathbb C)\), the matrices
\[
\{2^{-n/2}R:R\in\mathscr E_n\}
\]
form an orthonormal basis of \(M_{2^n}(\mathbb C)\).

Let \(x\in L_n\) and put \(M:=xx^\dagger\). Since \(R^\dagger=R\), we have
\[
\langle R,M\rangle_{\mathrm{HS}}
=\operatorname{Tr}(Rxx^\dagger)
=x^\dagger Rx
=q_R(x).
\]
Thus the coefficient of \(2^{-n/2}R\) in the orthonormal expansion of \(M\) is \(2^{-n/2}q_R(x)\), and consequently
\begin{equation}\label{eq:Pauli-expansion}
M=\frac1{2^n}\sum_{R\in\mathscr E_n}q_R(x)R.
\end{equation}

Applying Parseval's identity to \(M\) and using \(q_R(x)\in\mathbb R\), we obtain
\[
\frac1{2^n}\sum_{R\in\mathscr E_n}q_R(x)^2
=\langle M,M\rangle_{\mathrm{HS}}.
\]
On the other hand,
\[
\begin{aligned}
\langle M,M\rangle_{\mathrm{HS}}
&=\operatorname{Tr}(xx^\dagger xx^\dagger)\\
&=(x^\dagger x)\operatorname{Tr}(xx^\dagger)\\
&=(x^\dagger x)^2
=\|x\|^4.
\end{aligned}
\]
Consequently,
\begin{equation}\label{eq:Parseval-total}
\sum_{R\in\mathscr E_n}q_R(x)^2
=2^n\|x\|^4.
\end{equation}

For \(n\ge1\), subtracting the identity term \(q_I(x)^2=\|x\|^4\) gives
\[
\sum_{R\in\mathscr E_n\setminus\{I\}}q_R(x)^2
=(2^n-1)\|x\|^4.
\]
Since \(\mathscr E_n\setminus\{I\}\) has \(4^n-1=(2^n-1)(2^n+1)\) elements, we obtain the exact average
\begin{equation}\label{eq:Parseval-average}
\frac1{4^n-1}
\sum_{R\in\mathscr E_n\setminus\{I\}}q_R(x)^2
=\frac{\|x\|^4}{2^n+1}.
\end{equation}

\section{Construction of the Candidate Norms}
\label{sec:attainment}

We now prove that every norm listed in Theorem~\ref{thm:main-spectrum} is attained.
Using the affine-subspace definition \eqref{eq:Lambda-affine}, we construct each candidate norm as the squared norm of a sum of two minimal vectors, with their inner product determined by the intersection of their supports.

Fix integers \(m\ge4\) and \(j\) satisfying
\[
\left\lceil\frac m2\right\rceil\le j\le m-2,
\]
and put \(a:=m-j\). Then
\[
2\le a\le\left\lfloor\frac m2\right\rfloor,
\]
so in particular \(2a\le m\).

In \(V_m=\mathbb F_2^m\), choose the coordinate subspaces
\[
U:=\{t\in V_m:t_1=0\},
\qquad
V:=\{t\in V_m:t_2=\cdots=t_{2a}=0\}.
\]
Their codimensions are \(1\) and \(2a-1\), respectively.
Since their defining equations involve disjoint sets of coordinates, we have
\begin{equation}\label{eq:construction-intersection-dimension}
\dim(U\cap V)=m-2a=2j-m.
\end{equation}

Set
\[
p:=x_U,
\qquad
q:=-2^{a-1}x_V.
\]
Since
\[
\left\lfloor\frac{\operatorname{codim}U}{2}\right\rfloor=0,
\qquad
\left\lfloor\frac{\operatorname{codim}V}{2}\right\rfloor=a-1,
\]
it follows from \eqref{eq:Lambda-affine} that \(p,q\in\Lambda_m\). Their squared norms are
\[
\|p\|^2=|U|=2^{m-1}=d_m,
\qquad
\|q\|^2=2^{2a-2}|V|
=2^{2a-2}2^{m-2a+1}=d_m.
\]
Thus both \(p\) and \(q\) are minimal vectors of \(\Lambda_m\).

Using \((x_U,x_V)=|U\cap V|\) and \(\dim(U\cap V)=m-2a\), we obtain
\begin{equation}\label{eq:construction-inner-product}
\begin{aligned}
(p,q)
&=-2^{a-1}|U\cap V|\\
&=-2^{a-1}2^{m-2a}\\
&=-2^{m-a-1}
=-2^{j-1},
\end{aligned}
\end{equation}
where the last equality uses \(a=m-j\).

Since \(p,q\in\Lambda_m\), their sum belongs to \(\Lambda_m\), and
\begin{equation}\label{eq:attained-norm}
\begin{aligned}
\|p+q\|^2
&=\|p\|^2+\|q\|^2+2(p,q)\\
&=2d_m-2^j.
\end{aligned}
\end{equation}
As \(j\) was arbitrary in the range \(\lceil m/2\rceil\le j\le m-2\), this realizes every candidate norm in Theorem~\ref{thm:main-spectrum}.

\section{Exclusion of All Other Norms}
\label{sec:exclusion}

\begin{proof}[Proof of Theorem~\ref{thm:main-spectrum}]

Recall that \(d_m=2^{m-1}\) and \(\gamma_m=2^{\lceil m/2\rceil}\).
Section~\ref{sec:attainment} establishes that every listed norm is attained.
It remains to show that any \(x\in\Lambda_m\) satisfying
\[
d_m<\|x\|^2<2d_m
\]
has squared norm \(2d_m-2^j\) for some integer \(j\) with \(\lceil m/2\rceil\le j\le m-2\).
We prove this assertion by induction on \(m\).

For \(m=1\), we have \(\Lambda_1=\mathbb Z^2\) and \(d_1=1\).
Every squared norm in \(\Lambda_1\) is an integer, so no vector has squared norm strictly between \(d_1\) and \(2d_1\).

For \(m=2,3\), equation~\eqref{eq:Lambda-norm-divisibility} gives
\[
\|y\|^2\in d_m\mathbb Z
\qquad(y\in\Lambda_m),
\]
since \(2^{\lceil m/2\rceil}=2^{m-1}=d_m\) in these two cases.
Thus no squared norm lies strictly between \(d_m\) and \(2d_m\), completing the initial cases.

By Lemma~\ref{lem:Gaussian-realization}, the isometries \(\rho_{m-1}:L_{m-1}\to\Lambda_m\), \(m\ge1\), allow us to carry out the induction in the Gaussian models.
Let \(m\ge4\), and assume that the asserted norm restriction holds for \(L_{m-2}\).
Take \(x\in L_{m-1}\) satisfying
\[
d_m<\|x\|^2<2d_m.
\]
For this fixed \(x\), write \(q_R:=q_R(x)\) for \(R\in\mathscr E_{m-1}\setminus\{I\}\).
We distinguish two cases according to the parity of \(m\).

\medskip\noindent\textit{Case 1: odd \(m\).}

Assume first that \(m\) is odd.
Then \(\gamma_m=2^{(m+1)/2}\), so \(\gamma_m^2=4d_m\).
Applying \eqref{eq:Parseval-average} with \(n=m-1\), we obtain
\[
\frac{1}{d_m^2-1}
\sum_{R\in\mathscr E_{m-1}\setminus\{I\}}q_R^2
=\frac{\|x\|^4}{d_m+1}
<\frac{4d_m^2}{d_m+1}
<4d_m
=\gamma_m^2.
\]

By \eqref{eq:imbalance-divisibility}, every nonzero \(q_R\) has square at least \(\gamma_m^2\).
The preceding strict inequality therefore implies that \(q_R=0\) for some \(R\in\mathscr E_{m-1}\setminus\{I\}\).
Lemma~\ref{lem:recursive-block-norms} then gives \(u_R,v_R\in L_{m-2}\) with
\[
\|u_R\|^2=\|v_R\|^2=\frac{\|x\|^2}{2}.
\]
In particular,
\[
d_{m-1}=\frac{d_m}{2}
<\|u_R\|^2
<d_m=2d_{m-1}.
\]

Applying the induction hypothesis to \(u_R\), we obtain
\[
\|u_R\|^2=2d_{m-1}-2^k=d_m-2^k
\]
for some integer \(k\) satisfying
\[
\frac{m-1}{2}\le k\le m-3.
\]
Consequently,
\[
\|x\|^2=2\|u_R\|^2=2d_m-2^{k+1},
\qquad
\left\lceil\frac m2\right\rceil
=\frac{m+1}{2}\le k+1\le m-2,
\]
as required.

\medskip\noindent\textit{Case 2: even \(m\).}

Assume now that \(m\) is even.
Then \(\gamma_m=2^{m/2}\) and \(\gamma_m^2=2d_m\).

We first record a consequence of the induction hypothesis for the norms of vectors in the lower-dimensional lattice.
For every \(y\in L_{m-2}\) with \(\|y\|^2<d_m\), the quantity
\[
\frac{d_m-\|y\|^2}{\gamma_m}
\]
is a power of two with nonnegative integer exponent.
For \(y=0\), this quantity equals \(2^{m/2-1}\); for \(\|y\|^2=d_m/2\), it equals \(2^{m/2-2}\).
Both exponents are nonnegative since \(m\ge4\).
In all remaining cases, \(\min(L_{m-2})=d_m/2\) gives
\[
d_m/2<\|y\|^2<d_m.
\]
The induction hypothesis therefore yields \(\|y\|^2=d_m-2^k\) for some integer \(k\) with \(m/2\le k\le m-3\), and hence
\[
\frac{d_m-\|y\|^2}{\gamma_m}=2^{k-m/2},
\]
which proves the assertion.

We now apply this observation to the coordinate subvectors obtained from \(x\) by the lattice automorphisms.
By Lemma~\ref{lem:Ln-norm-divisibility}, \(\|x\|^2\in\gamma_m\mathbb Z\), so we may write
\[
\|x\|^2=2d_m-\gamma_m T
\]
for some integer \(T\).
The inequalities \(d_m<\|x\|^2<2d_m\) give
\[
1\le T<\frac{d_m}{\gamma_m}=\frac{\gamma_m}{2}.
\]

For each \(R\in\mathscr E_{m-1}\setminus\{I\}\), set
\[
z_R:=\frac{q_R}{\gamma_m}\in\mathbb Z,
\]
where integrality follows from \eqref{eq:imbalance-divisibility}.

For each \(R\in\mathscr E_{m-1}\setminus\{I\}\), Lemma~\ref{lem:recursive-block-norms} gives coordinate subvectors \(u_R,v_R\in L_{m-2}\) whose smaller and larger squared norms are, respectively,
\[
\frac{\|x\|^2-|q_R|}{2}
=d_m-\gamma_m\frac{T+|z_R|}{2},
\qquad
\frac{\|x\|^2+|q_R|}{2}
=d_m-\gamma_m\frac{T-|z_R|}{2}.
\]
The smaller squared norm is always less than \(d_m\), while the larger one is less than \(d_m\) precisely when \(|z_R|<T\).
Applying the preceding observation to those coordinate subvectors whose squared norms are less than \(d_m\), we conclude that
\[
\frac{T+|z_R|}{2}
\quad\text{and, whenever }|z_R|<T,\quad
\frac{T-|z_R|}{2}
\]
are powers of two with nonnegative integer exponents.

There are \(4^{m-1}-1=d_m^2-1\) such indices.
Applying \eqref{eq:Parseval-average} with \(n=m-1\) and using \(\|x\|^2=\gamma_m(\gamma_m-T)\), we obtain
\begin{equation}\label{eq:z-average}
\frac1{d_m^2-1}\sum_{R\in\mathscr E_{m-1}\setminus\{I\}}z_R^2
=\frac{(\gamma_m-T)^2}{d_m+1}
<\frac{2d_m}{d_m+1}
<2,
\end{equation}
where the strict inequalities follow from \(0<T<\gamma_m\) and \(\gamma_m^2=2d_m\).
Hence there exists \(R_0\in\mathscr E_{m-1}\setminus\{I\}\) such that \(z_{R_0}^2<2\).
As \(z_{R_0}\) is an integer, it follows that
\[
|z_{R_0}|\in\{0,1\}.
\]
If \(z_{R_0}=0\), the condition on \((T+|z_{R_0}|)/2\) shows that \(T/2\) is a power of two with nonnegative integer exponent, and therefore \(T\) is also a power of two.

Suppose now that \(|z_{R_0}|=1\).
If \(T=1\), then \(T\) is already a power of two.
If \(T>1\), then \(|z_{R_0}|=1<T\), so both conditions apply and
\[
\frac{T-1}{2}
\quad\text{and}\quad
\frac{T+1}{2}
\]
are powers of two with nonnegative integer exponents.
Since they differ by \(1\), they must equal \(1\) and \(2\), respectively, giving \(T=3\).
Thus, in either branch, \(T\) is a power of two or \(T=3\).

It remains to exclude \(T=3\).
Suppose that this value occurs.
For every \(R\in\mathscr E_{m-1}\setminus\{I\}\), the condition on \((T+|z_R|)/2\) implies that \(3+|z_R|\) is a power of two.
Thus \(3\nmid z_R\), and hence \(z_R^2\equiv1\pmod3\).
Since \(m\) is even, \(d_m=2^{m-1}\equiv-1\pmod3\), so
\[
\sum_{R\in\mathscr E_{m-1}\setminus\{I\}}z_R^2
\equiv d_m^2-1
\equiv0\pmod3.
\]
On the other hand, the equality in \eqref{eq:z-average} gives
\[
\sum_{R\in\mathscr E_{m-1}\setminus\{I\}}z_R^2
=(d_m-1)(\gamma_m-3)^2
\equiv(d_m-1)\gamma_m^2
\equiv1\pmod3,
\]
where we used \(\gamma_m^2=2d_m\).
This contradiction rules out \(T=3\).

Consequently, \(T=2^a\) for some nonnegative integer \(a\).
Since
\[
2^a=T<\frac{\gamma_m}{2}=2^{m/2-1},
\]
we have \(0\le a\le m/2-2\).
Substituting into the expression for \(\|x\|^2\), we obtain
\[
\|x\|^2=2d_m-\gamma_m T=2d_m-2^{m/2+a},
\]
where
\[
\left\lceil\frac m2\right\rceil
=\frac m2\le\frac m2+a\le m-2.
\]
This proves the assertion in the even case.

The two cases discussed above complete the induction.
Together with the constructions in Section~\ref{sec:attainment}, this proves Theorem~\ref{thm:main-spectrum}.

\end{proof}

\section{Conclusion}
\label{sec:conclusion}

We have determined the complete norm spectrum of the Barnes--Wall lattices strictly between the minimum and twice its value.
To conclude, the proof of Theorem~\ref{thm:main-spectrum} develops the recursive approach of Kliuchnikov and Sch\"onnenbeck beyond the minimal vectors.
Lattice automorphisms arising from Clifford unitaries transform a vector into
coordinate subvectors in a lower-dimensional lattice, while 2-divisibility and a matrix Parseval identity impose
arithmetic and averaging constraints on their norms.
Together with induction on the dimension, these constraints
exclude all other norms in the stated interval.

Further work could investigate whether the recursive methods used here can constrain the norm spectrum at or above twice the minimum.
Another question is whether these methods, together with the explicit constructions, can be adapted to describe the orbits of vectors of each admissible norm under lattice automorphisms.

\section{Declaration of Generative AI and AI-Assisted Technologies in the Manuscript Preparation Process}
During the preparation of this work, the authors used GPT-5.6 to assist with proof development and GPT-6 to assist with manuscript preparation under the authors’ direction. After using these tools, the authors reviewed and edited the content as needed and take full responsibility for the content of the article.

\end{document}